\documentclass[11pt]{article}
\usepackage{amscd}
\usepackage{amsfonts}
\usepackage[all]{xy}
\usepackage{amsmath,amsthm,hyperref}
\usepackage{amsmath,amssymb,amsthm,latexsym}
\usepackage{amscd}

\newtheorem{theorem}{Theorem}[section]
\newtheorem{proposition}[theorem]{Proposition}
\newtheorem{definition}[theorem]{Definition}

\newtheorem{lemma}[theorem]{Lemma}

\theoremstyle{remark}
\newtheorem{remark}[theorem]{Remark}
\newtheorem{example}[theorem]{\bf Example}
\newcommand{\R}{\mathbb{R}}

\def\R{\mathbb R}
\def\o{\omega}
\def\O{\Omega}
\def\a{\alpha}
\def\e{\varepsilon}
\def\rightarrow{\to}
\def\<{\langle}
\def\>{\rangle}
\def\t{\theta}
\begin{document}
\title{\bf{On transforms of timelike isothermic surfaces in pseudo-Riemannian space forms }}
\author{Yuping Song, \ \
        Peng Wang\footnote {corresponding author, supported by the Fundamental Research Funds for the Central Universities and the Project 11201340 of NSFC.} }
\date{}
\maketitle

\begin{center}
{\bf Abstract}
\end{center}
  The basic theory on the conformal geometry of timelike surfaces in pseudo-Riemannian space forms is introduced, which is parallel to the classical framework of Burstall etc. for spacelike surfaces. Then we provide a discussion on the transforms of timelike $(\pm)-$isothermic surfaces (or real isothermic, complex isothermic surfaces), including $c-$ polar transforms, Darboux transforms and spectral transforms. The first main result is that $c-$polar transforms preserve timelike $(\pm)-$isothermic surfaces, which are generalizations of the classical Christoffel transforms. The next main result is that a Darboux pair of timelike  isothermic surfaces can also be characterized as a Lorentzian $O(n-r+1,r+1)/O(n-r,r)\times O(1,1)-$type curved flat. Finally two permutability theorems of $c-$polar transforms are established.\\



\section{Introduction}

The study of isothermic surfaces in $\mathbb{R}^3$ can be tracked back to the early develop of the classical differential geometry.
It was revealed only 30 years ago in \cite{CGS} that,
there is a structure of integrable system underlying the theory
about isothermic surfaces, which has its root in the transforms of isothermic surfaces. This theory was explained later as a special kind of curved flat in \cite{BHPP} and as a $G^1_{m,1}-$system in \cite{BDPT} (see also \cite{Terng}).  For a complete theory concerning the geometry and integrable system of isothermic surfaces in $\mathbb{R}^n$ as well as the interpretation that what is an integrable system, we refer to  \cite{Bur}.

In recent years, the integrable system theory of isothermic surfaces is used to the study transforms of timelike isothermic surfaces in pseudo-Riemannian space forms (\cite{Dussan-M2005}, \cite{Dussan-M2006}, \cite{Mag}, \cite{ZCC}), mainly using the methods
developed by Burstall in \cite{Bur} and Bruck et al in
\cite{BDPT}.

On the other hand, a geometric observation initiated in the works of \cite{Ma-W2}, \cite{Wangpeng2}, where it was shown that there exists a new kind of
 natural transforms for spacelike surfaces in n-dimensional Lorentzian space
forms. The key observation is that any parallel normal section of the conformal flat normal bundle yields a conformal map.
 So we can define the polar transforms
locally. The new isothermic surface produced in this way is neither
the spectral transform nor the Darboux transform of the original isothermic surface. It is the generalization of the Christoffel transforms for isothermic surfaces.

We also note that in \cite{Smyth}, the globally isothermic surfaces, named also as soliton surfaces, appeared in the mechanical equilibrium of closed membranes. Since timelike surfaces arise naturally in the study of surfaces in Lorentzian space forms, it does make sense to have a detailed study on the geometry of timelike isothermic surfaces in pseudo-Riemannian space forms. Moreover, some new examples of global isothermic surfaces are also presented in our paper, see Section 2.3.

In this paper, we want to provide a similar theory concerning the conformal geometry and integrable system of timelike isothermic surfaces in pseudo-Riemannian space forms, following the treatment in \cite{BPP}, \cite{Bur} and \cite{BHPP}.
For this purpose, we first derive the surfaces theory for timelike surfaces in pseudo-Riemannian space forms. Then we unify these into the conformal geometry of timelike surfaces in the projective light cone, following the framework in \cite{BPP}. After these preparation, we define the $c-$polar transforms of timelike $(\pm)-$isothermic surfaces and prove that  $c-$polar surfaces are also timelike $(\pm)-$isothermic surfaces when they are immersions. Moreover, we derive a geometric description of the Darboux transforms for timelike $(\pm)-$isothermic surfaces. As an application, we show that a Darboux pair of timelike $(\pm)-$isothermic surfaces is equivalent to a Lorentzian $O(n-r+1,r+1/O(n-r,r)\times O(1,1))-$ curved flat.
Since $c-$polar transforms can be viewed as generalized Christoffel transforms, it is natural to expect that $c-$polar
transforms commute with the spectral transform and the Darboux
transform. Similar to \cite{Ma-W2}, we obtain two of such
\emph{permutability theorems}, see Section 5.

This paper is organized as follows. In Section~2  we
 derive the main theory of timelike surfaces both in the isometric case and in the conformal case, together with the relations between them. The definition and basic properties of
isothermic surfaces are also discussed here. A Bonnet-type theorem is also provided. Then we introduce
the $c-$polar transforms of timelike isothermic surfaces and the description of them in Section~3.
In Section 4 we prove that a
Darboux pair of timelike isothermic surfaces are equivalent to a $O(n-r+1,r+1/O(n-r,r)\times O(1,1))-$ curved flat.
Finally, we provide a brief proof of the permutability theorems between $c-$polar transforms and the spectral transforms, the Darboux
transforms in Section~5.

\section{Surface Theory for timelike isothermic surface}

In this section we first review the timelike surface theory in space forms. Then we consider the conformal geometry of timelike surface theory
in pseudo-Riemannian space forms. Furthermore, the relations between the isometric invariants and the
conformal invariants of a timelike surface are established. Finally we end this section by providing some new examples of timelike isothermic surfaces.

\subsection{Timelike isothermic surfaces in $N^n_r(c)$ }

Let $\R^m_s$ be the linear space $\R^m$ equipped with the quadric form
\begin{align*}
&\langle x,x\rangle=\sum_1^{m-s}x_i^2-\sum_{m-s+1}^{m}x_i^2.
\end{align*}
Let $N^n_r(c)$ be the $n$-dimensional pseudo-Riemannian space form with constant
curvature $c,c=0,1,-1$, with its pseudo-Riemannian metric of signature $(n-r,r),n-r\geq 2,r\geq1.$
It is well-known that $N^n_r(c)$ can be described as
\begin{equation*}
\left\{\begin {array}{lllll}
x\in N^n_r(c)=\R^n_r,\hskip 5pt c=0,\\
x\in N^n_r(c)=S^n_r\subset \R^{n+1}_r, \hskip 5pt c=1,\\
x\in N^n_r(c)=H^n_r\subset \R^{n+1}_r, \hskip 5pt c=-1,
\end {array}\right.
\end{equation*}

We recall the basic properties of timelike surfaces. A surface $x:M\rightarrow N^n_r(c)$ is called a timelike surface or Lorentzian surface if the induced metric $x^*\langle,\rangle$ on $M$ is a Lorentzian metric. For the two dimensional Lorentzian surface $x$, there are two (conformally) invariant lightlike directions and integration along them gives an (locally) asymptotic coordinate system on $M$.  Now let  $x:M\rightarrow N^n_r(c)$ be an oriented timelike surface with local asymptotic coordinate $(u,v)$ and metric $|dx|^2=e^{2\omega}dudv$.  Choose an orthonormal frame  $\{n_\alpha\}$ of the normal bundle with $\langle n_\alpha,n_\alpha\rangle=\varepsilon_{\alpha}=\pm1$, and let $D_u$ and $D_v$ denote the normal connection.

The structure equations are as follows
\begin{equation}\label{eq-moving0}
\left\{\begin {array}{lllll}
x_{uu}=2\o_u x_u+\O_1,\\
x_{vv}=2\o_v x_v+\O_2,\\
x_{uv}=\frac{1}{2}e^{2\o}H-\frac{c}{2}e^{2\o}x,\\
(n_{\a})_u=D_u n_{\a}-\langle n_{\a},H\rangle x_u-2e^{-2\o}\langle n_{\a},\O_1\rangle x_v,\\
(n_{\a})_v=D_v n_{\a}-\langle n_{\a},H\rangle x_v-2e^{-2\o}\langle n_{\a},\O_2\rangle x_u,
\hskip 5pt \a=3,...,n.
\end {array}\right.
\end{equation}
Here  $H=\sum_\alpha h_\alpha n_\alpha$ is the mean curvature vector and the two $\Gamma(T^{\perp}M)-$valued 2-forms $\Omega_1du^2=\sum_{\alpha=3}^{n}\Omega_{1\alpha}n_{\alpha}dz^2$ and $\Omega_1du^2=\sum_{\alpha=3}^{n}\Omega_{1\alpha}n_{\alpha}dz^2$ are the vector-valued {\em Hopf differentials}. The Gauss equation, Codazzi equations and Ricci equations as integrable equations are as below
\begin{equation}\label{eq-integrable0}
\left\{\begin {array}{lllll}
\langle H,H\rangle-K+c=4e^{-4\o}\langle \O_1,\O_2\rangle,\\
D_u H=2e^{-2\o}D_v\O_1,\hskip 5pt D_v H=2e^{-2\o}D_u\O_2,\\
R^D_{uv}n_{\a}:=D_v D_u n_{\a}-D_v D_u n_{\a}=2e^{-2\o}(\langle n_{\a},\O_1\rangle \O_2-\langle n_{\a},\O_2\rangle \O_1).
\end {array}\right.
\end{equation}

There are several equivalent definition of timelike isothermic surfaces, concerning different properties of them. Here first we define the $(\pm)-$timelike isothermic surfaces following the terms of Fujioka and Inoguchi.
\begin{definition}\cite{F-I}
Let $x:M\rightarrow N^{n}_{r}(c)$ be a timelike surface. It is called $(\pm)-$isothermic if around each
point of $M$ there exists an asymptotic coordinate $(u,v)$ such that the vector-valued Hopf differentials $\Omega_1=\pm\Omega_2$. Such coordinate $(u,v)$ is called an adapted (asymptotic) coordinate.
\end{definition}

Another equivalent definition is given by Dussan and Magid in \cite{Dussan-M2005}, the so-called {\em timelike real isothermic and complex isothermic surfaces}.
\begin{definition}\cite{Dussan-M2005}
Let $x:M\rightarrow N^{n}_{r}(c)$ be a timelike surface. Suppose that the normal bundle of $x$ is flat and that $n_{\alpha}$ is a parallel orthonormal frame of the normal bundle. Then
$x$ is called real isothermic surface, if there exist coordinates  $(\tilde{u},\tilde{v})$, such that its first and second fundamental forms is of the form
\begin{equation}
I=e^{2\omega}(d\tilde{u}^2-d\tilde{v}^2),\ II=\sum_{\alpha}(b_{1\alpha}d\tilde{u}^2-b_{2\alpha} d\tilde{v}^2)n_{\alpha}.\end{equation}
And $x$ is called complex isothermic surface, if there exist coordinates $(\tilde{u},\tilde{v})$, such that the first and second fundamental forms of $x$ is of the form
\begin{equation}
I=e^{2\omega}(d\tilde{u}^2-d\tilde{v}^2),\ II=\sum_{\alpha}(b_{1\alpha}(d\tilde{u}^2-d\tilde{v}^2)+b_{0\alpha} d\tilde{v}^2)n_{\alpha}.\end{equation}
\end{definition}

From the structure equations above, it is direct to see the equivalence between these two definitions.
\begin{lemma} For a timelike surface, it is $(+)-$isothermic if and only if it is real isothermic. It is $(-)-$isothermic if and only if it is complex isothermic.
\end{lemma}

\begin{proof}$\O_1=\pm\O_2$ together with the conformal Ricci equations in (2) shows that the normal bundle
of $x$ is flat. This is an important property of isothermic surfaces, which guarantees that all shape
operators commute and the curvature lines could still be defined. Setting $u=s+t$, $v =s-t$, the
two fundamental forms of an isothermic surface, with respect to some parallel normal frame $\{n_{\a}\}$,
are of the form
\begin{align}
I=e^{2\o}(ds^2-dt^2),\hskip 5pt II=\sum_{\a}(b_{\a 1}ds^2-b_{\a 2}dt^2)n_{\a}
\end{align}
if $x$ is $(+)-$ isothermic and
\begin{align}
I=e^{2\o}(ds^2-dt^2),\hskip 5pt II=\sum_{\a}(b_{\a 1}(ds^2-dt^2)-2b_{\a 2}dsdt)n_{\a}
\end{align}
if $x$ is $(-)-$ isothermic.
\end{proof}
\par

Since timelike isothermic surfaces are also conformally invariant as the spacelike cases, it will be better to deal with them by using a conformally invariant frame. This idea leads the next subsection.

\subsection{Timelike isothermic surfaces in $Q^n_r$ }
Let $C^{m-1}_s$ be the light cone of $\R^m_s$. Here $\R^m_s$ is the same space defined in Section 2.1. The metric of $\R^{n+2}_{r+1}$ induces a $(n-r,r)-$ type pseudo-Riemannian metric on
\begin{align*}
S^{n-r}\times S^r=\{x\in\R^{n+2}_{r+1}|\sum_{i=1}^{n-r+1}x_i^2=\sum_{i=n-r+2}^{n+2}x_i^2=1\}\subset C^{n+1}_r\setminus\{0\},
\end{align*}
which can be written as $g(S^{n-r})\oplus(-g(S^r))$ with $g(S^{n-r})$ and $g(S^r)$ the standard Riemannian metrics on $S^{n-r}$ and $S^r$.
Let
\begin{align*}
Q^n_r=\{[x]\in\R P^{n+1}|x\in C^{n+1}_r\setminus\{0\}\}
\end{align*}
be the projective light cone. The projection $\pi:S^{n-r}\times S^r\to Q^n_r$ therefore induces a conformal structure $[h]$ on $Q^n_r$, with $h$ locally a $(n-r,r)-$type pseudo-Riemannian metric.
The conformal group of $(Q^n_r,[h])$ is the orthogonal
group $O(n-r+1,r+1)$ which keeps the inner product of $\R^{n+2}_{r+1}$
invariant and acts on $Q^n_r$ by
\begin{equation}\label{eq-action}T([x])=[xT],\ T\in O(n-r+1,r+1).\end{equation}
The space forms $N^n_r(c)$, $c=0,1,-1$ can be conformally embedded as a proper subset of $Q^n_r$ via
\begin{align}
x\mapsto \left(\frac{\langle x,x\rangle-1}{2},x,\frac{\langle x,x\rangle+1}{2}\right).
\end{align}

Basic conformal surface theory shows that for a timelike surface $y:M\to Q^n_r$ with local
asymptotic coordinate $(u,v)$, there exists a local lift $Y:U\to C^{n+1}_r\setminus\{0\}$
of $y$ such that $\langle dY,dY\rangle=\frac{1}{2}(du\otimes dv+dv\otimes du)$ in an open
subset $U$ of $M$. We denote by
\begin{align}
V=\mathrm{Span}\{Y,Y_u,Y_v,Y_{uv}\}
\end{align}
the conformal tangent bundle, and by $V^{\bot}$ its orthogonal complement
or the conformal normal bundle. Set $ N=2Y_{uv}+2\langle \kappa_1,\kappa_2\rangle Y ,$
we have $$\langle Y, N\rangle=-1,\ \langle Y_u, N\rangle= \langle Y_v, N\rangle=\langle N, N\rangle=0.$$
Let $D$ denote the normal connection and let $\psi\in\Gamma(V^{\bot})$ be
an arbitrary section. Then the structure equation of $y$ can be derived as follows \cite{Wangpeng}
\begin{equation}\label{eq-moving1}
\left\{\begin {array}{lllll}
Y_{uu}=-\frac{s_1}{2}Y+\kappa_1,\\
Y_{vv}=-\frac{s_2}{2}Y+\kappa_2,\\
Y_{uv}=-\langle \kappa_1,\kappa_2\rangle Y+\frac{1}{2}N,\\
N_u=-2\langle \kappa_1,\kappa_2\rangle Y_u-s_1Y_v+2D_v \kappa_1,\\
N_v=-s_2Y_u-2\langle \kappa_1,\kappa_2\rangle Y_v+2D_u \kappa_2,\\
\psi_u=D_u\psi+2\langle\psi,D_v \kappa_1\rangle Y-2\langle \psi,\kappa_1\rangle Y_v,\\
\psi_v=D_v\psi+2\langle\psi,D_u \kappa_2\rangle Y-2\langle \psi,\kappa_2\rangle Y_u.\\
\end {array}\right.
\end{equation}
The first two equations above define four basic invariants $\kappa_1$, $\kappa_2$ and $s_1$, $s_2$ dependent on
$(u,v)$. Similar to the
case in M\"{o}bius
geometry, $k_i$, and $s_i$ are called the conformal Hopf differential and the Schwarzian
derivative of $y$, respectively (see  \cite{BPP}, \cite{Ma1} and \cite{Wangpeng}).
\par
The conformal Hopf differential defines a conformal invariant metric
\begin{equation}
g:=\langle \kappa_1,\kappa_2\rangle dudv
\end{equation}
and the Willmore functional
\begin{equation}
W(y):=\frac{1}{2}\int_M\langle \kappa_1,\kappa_2\rangle dudv
\end{equation}
For more details on timelike Willmore surfaces we refer to \cite{DWang}.
The conformal Gauss Codazzi and Ricci equations as integrable conditions are:
\begin{equation}\label{eq-inte1}
\left\{\begin {array}{lllll}
\frac{(s_1)_v}{2}=3\langle D_u \kappa_2,\kappa_1\rangle+\langle D_u \kappa_1,\kappa_2\rangle,\\
\frac{(s_2)_u}{2}=\langle \kappa_1, D_v \kappa_2\rangle+3\langle D_v \kappa_1,\kappa_2\rangle,
\end {array}\right.
\end{equation}
\begin{equation}\label{eq-inte2}
D_v D_v \kappa_1+\frac{s_2 \kappa_1}{2}=D_u D_u \kappa_2+\frac{s_1 \kappa_2}{2}
\end{equation}
\begin{equation}\label{eq-inte3}
R^D_{uv}\psi:=D_v D_u\psi-D_u D_v\psi=2\langle \psi,\kappa_1\rangle \kappa_2-2\langle \psi,\kappa_2\rangle \kappa_1.
\end{equation}

Let $(\tilde{u}=\tilde{u}(u),\tilde{v}=\tilde{v}(v))$ be another asymptotic coordinate with $\langle Y_{\tilde{u}},Y_{\tilde{v}}\rangle=\frac{1}{2}e^{2\rho}$,
then $\tilde{Y}=e^{-\rho}Y$, and correspondingly,
\begin{equation}\tilde\kappa_1\frac{d\tilde{u}^2}{ \sqrt{d\tilde{u}d\tilde{v}} }=\kappa_1\frac{du^2}{ \sqrt{dudv} },\ \ \tilde\kappa_2\frac{d\tilde{v}^2}{ \sqrt{d\tilde{u}d\tilde{v}} }=\kappa_2\frac{dv^2}{ \sqrt{dudv} }.
\end{equation}
And
\begin{equation}
\tilde{s}_1=s_1(\frac{\partial u}{\partial\tilde{u}})^2+2\rho_{\tilde{u}\tilde{u}}-2(\rho_{\tilde{u}})^2,
\end{equation}
\begin{equation}
\tilde{s}_2=s_2(\frac{\partial v}{\partial\tilde{v}})^2+2\rho_{\tilde{v}\tilde{v}}-2(\rho_{\tilde{v}})^2,
\end{equation}

With these at hand, we can state the conformal version of Bonnet theorem on the existence and uniqueness of conformally timelike surface as follows. (This is directly derived as the usual submanifold theory. Since we can not find a reference for Bonnet theorem of timelike surfaces, we write down it here for convenience.):
\begin{theorem}\label{Bonnet} For an oriented timelike surface $M$ with local local
asymptotic coordinate $(u,v)$, the differentials $\frac{du^2}{ \sqrt{dudv}} $ and  $\frac{dv^2}{ \sqrt{dudv}} $ define two real line bundles $L_1$ and $L_2$ over $M$.

For a timelike surface $M$ with local
asymptotic coordinate $(u,v)$, let $NM$ be an $n-2$ dimensional flat sub-vector bundle of the trivial bundle $M\times \mathbb{R}^{n+2}_{r+1}$ over $M$,   with an induced $(n-1,r-1)-$type flat pseudo Riemannian metric on each fiber of $NM$. Denote by $D_{u}$ and $D_v$ the connection over $NM$. Chose a section $\kappa_1\in \Gamma(NM\otimes L_2)$ and a section $\kappa_2\in \Gamma(NM\otimes L_2)$, and let $s_1$, $s_2$ be Schwarzians defined on $M$. If $\kappa_1,\ \kappa_2, \ s_1, \ s_2, D_u, D_v$ satisfy the integrable equations \eqref{eq-inte1}, \eqref{eq-inte2},\ \eqref{eq-inte3}, then there exists a conformal immersion $y:M\rightarrow Q^n_r$ with conformal data $\{\kappa_1,\ \kappa_2, \ s_1, \ s_2, D_u, D_v, NM\}$ as given.

Moreover, let $y,\hat{y}:M\to Q^n_r$ be two timelike surfaces with (same) local
asymptotic coordinate $(u,v)$. Assume that for some canonical lift $Y$ of $y$ and $\hat{Y}$ of $\hat{y}$:

(i). $Y$ and $\hat{Y}$ have the same Schwarzians, that is, $s_1=\hat{s}_1$, $s_2=\hat{s}_2$.

(ii).There exist normal frames $\{\phi_{\alpha}\}$ of $y$ and normal frames $\{\hat\phi_{\alpha}\}$ of $\hat{y}$, such that for these frames, $$D_u\psi_{\alpha}=D_{u}\hat{\psi}_{\alpha},\ D_v\psi_{\alpha}=D_{v}\hat{\psi}_{\alpha}.$$

(iii). Under the frames (ii), the Hopf differential are of the form
$$\kappa_1=\sum_{\alpha}k_{1\alpha}\psi_{\alpha},\ \kappa_2=\sum_{\alpha}k_{2\alpha}\psi_{\alpha},\ \hat{\kappa}_1=\sum_{\alpha}k_{1\alpha}\hat{\psi}_{\alpha},\ \hat{\kappa}_2=\sum_{\alpha}k_{2\alpha}\hat{\psi}_{\alpha}.$$
Then there is a M\"{o}bius transformation (See \eqref{eq-action}) $T\in O(n+2,r+1):Q^n_r\rightarrow Q^n_r$, such that  $Ty=\hat{y}.$
\end{theorem}

Now let us turn to timelike isothermic surfaces. A timelike surface  $y:M\to Q^n_r$ is called $(\pm)-$isothermic if its Hopf differentials satisfy that
$\kappa_1=\pm \kappa_2$ with respect to some asymptotic coordinate $(u,v)$.
\par
To show the equivalence of this definition with the previous ones, we need to derive the relations between the isothermic invariants and conformal invariant of
a timelike surface $x:M\to N_r^n(c)$. We verify this for the case $c=0$. The other cases are omitted.
Let
\begin{equation}
X=\left(\frac{\langle x,x\rangle -1}{2},x,\frac{\langle x,x\rangle +1}{2}\right).
\end{equation}
Then $y=[X]$ is the corresponding surface into $Q^n_r$. It is direct to obtain a canonical lift
$Y=e^{-\o}X$ with respect to $(u,v)$. So
\begin{equation}\label{eq-trans1}
\left\{\begin {array}{lllll}
Y &=\frac{1}{2}e^{-\o}(\langle x,x\rangle -1,2x,\langle x,x\rangle +1),\\
Y_u&=-\o_u Y+e^{-\o}(\langle x_u,x\rangle,x_u,\langle x_u,x\rangle),\\
Y_v&=-\o_v Y+e^{-\o}(\langle x_v,x\rangle,x_v,\langle x_v,x\rangle)\\
N&=e^{\o}(\langle H,x\rangle+1,H,\langle H,x\rangle+1)\\
& \ \ -2\o_u Y_v-2\o_v Y_u+2(\langle \kappa_1,\kappa_2\rangle-\o_{uv}-\o_u\o_v)Y,\\
\psi_{\a}&=e^{-\o}(\langle n_{\a},x\rangle,n_{\a},\langle n_{\a},x\rangle)+\langle n_{\a},H\rangle Y.
\end {array}\right.
\end{equation}
Then we see that the normal connection of $\psi_{\a}$ is the same as the normal
connection of $\{n_{\a}\}$. We also calculate the Schwarzian derivatives and conformal differentials as
\begin{equation}\label{eq-trans2}
\left\{\begin {array}{lllll}
s_1&=2\o_{uu}-2\o_u^2+2\langle \O_1,H\rangle,\\
s_2&=2\o_{vv}-2\o_v^2+2\langle \O_2,H\rangle,\\
\kappa_1&=e^{-\o}(\langle \O_1,x\rangle,\O_1,\langle \O_1,x\rangle)+\langle \O_1,H\rangle Y,\\
\kappa_2&=e^{-\o}(\langle \O_2,x\rangle,\O_2,\langle \O_2,x\rangle)+\langle \O_2,H\rangle Y.
\end {array}\right.
\end{equation}
Therefore $\kappa_1=\pm \kappa_2$ if and only if $\O_1=\pm\O_2$. So a timelike
$(\pm)-$ isothermic surface in $N^n_r(c)$ can be conformally embedded into $Q^n_r$
and must be a conformal timelike
$(\pm)-$ isothermic surface in $Q^n_r$.

\subsection{Examples}

\begin{example}Let $x$ be a timelike $H=0$ surface or a  CMC $H$ timelike surface in a 3-dimensional Lorentzian space form. If $<\Omega_1,\Omega_2>\neq0$, then $x$ is isothermic (compare \cite{Ma1}, \cite{Hong}, \cite{In}, \cite{Wangpeng}). To see this, from the codazzi equation of $x$ in (2), we have
$$D_v\Omega_1=0,\ ~~~~~~D_u \Omega_2=0.$$
So if we assume that  $\Omega_1=a_1n$, $\Omega_2=a_2n$. Then $a_1a_2\neq0$ and $a_1=a_1(u)$, $a_2=a_2(v)$. So by suitable coordinate changing $u\mapsto \tilde{u}=\tilde{u}(u)$ and $v\mapsto \tilde{v}=\tilde{v}(v)$, one will have the new Hopf differential $$\tilde{\Omega}_1=\pm\tilde{\Omega}_2=\pm n.$$

Note that if $<\Omega_1,\Omega_2>=0$, $x$ may not be isothermic. One can find counterexamples from the null scrolls in \cite{F-I}.
\end{example}

\begin{example} Timelike surfaces by rotation on a timelike plane:
Let $$\gamma=(\gamma_1, \gamma_2, \cdots,\gamma_{n-1},0)$$ be a timelike curve in $\R^n_r$ with $\gamma_1\neq 0$. Then a similar procedure as \cite{Ma-W2} verifies that
$$x_{\gamma}=\left(\gamma_1\cosh\theta,\gamma_2,\gamma_3,\cdots,\gamma_{n-1},\gamma_1\sinh\theta \right)$$
is a timelike $(+)-$isothermic surface.

\end{example}

\begin{example} Timelike surfaces by rotation on a spacelike plane:
Let $\gamma=(0,\gamma_2,\gamma_3,\cdots,\gamma_n)$ be a timelike curve in $\R^n_r$ with $\gamma_2\neq 0$. Then a similar procedure as \cite{Ma-W2} verifies that
$$x_{\gamma}=\left(\gamma_2\cos\theta,\gamma_2\sin\theta,\gamma_3,\cdots,\gamma_n\right)$$
is a timelike $(+)-$isothermic surface.
\end{example}

\begin{example} \cite{Ma-W3}
Set
$$e_{1}=(\cos\frac{t\theta}{\sqrt{1-t^{2}}}\cos\phi,
\cos\frac{t\theta}{\sqrt{1-t^{2}}}\sin\phi,
\sin\frac{t\theta}{\sqrt{1-t^{2}}}\cos\phi,
\sin\frac{t\theta}{\sqrt{1-t^{2}}}\sin\phi),$$
$$e_{2}=\frac{\partial e_{1}}{\partial\phi}=e_{1\phi},~
e_{3}=\frac{\sqrt{1-t^{2}}}{t}e_{1\theta},~
e_{4}=\frac{\sqrt{1-t^{2}}}{t}e_{2\theta},$$ with $t<1$. Let
\begin{equation*} \label{eq:122}
\begin{split}
Y_{t}(\theta,\phi)&: \mathbb{R}\times\mathbb{R}\longrightarrow \mathbb{R}^{6}_{2} \\
 Y_{t}(\theta,\phi)&=(e_{1},\cos\frac{\theta}{\sqrt{1-t^{2}}},\sin\frac{\theta}{\sqrt{1-t^{2}}}).
 \end{split}
 \end{equation*}
For simplicity, we omit the subscript $"_{t}"$ of $Y_{t}$. We have
that $y=[Y]:\mathbb{R}\times\mathbb{R}\rightarrow Q^{4}_{1}$ is a
timelike $(-)-$isothermic Willmore cylinder of $Q^{4}_{1}$ for any $|t|<1$ and it is a torus when $t$ is rational.

Note that $y_t$ is also a conformally homogeneous surface in $Q^4_1$. We refer to \cite{Ma-W3} for details.
\end{example}

\begin{example} \cite{Ma-W3} Set
\begin{equation*}
\begin{split}
e_{1}=&(\cosh\frac{t\theta}{\sqrt{1-t^{2}}}\cosh\phi,
\sinh\frac{t\theta}{\sqrt{1-t^{2}}}\sinh\phi,0,\\
 &  \cosh\frac{t\theta}{\sqrt{1-t^{2}}}\sinh\phi,
\sinh\frac{t\theta}{\sqrt{1-t^{2}}}\cosh\phi,0),\\
  e_{2}=&\frac{\partial e_{1}}{\partial\phi}=e_{1\phi},~
e_{3}=\frac{\sqrt{1-t^{2}}}{t}e_{1\theta},~
e_{4}=\frac{\sqrt{1-t^{2}}}{t}e_{2\theta},\\
f_{1}= &
(0,0,\sinh\frac{\theta}{\sqrt{1-t^{2}}},0,0,\cosh\frac{\theta}{\sqrt{1-t^{2}}}),f_2=\sqrt{1-t^{2}}f_{1\theta}.
\end{split}
\end{equation*}
with $0<t<1$. Let
\begin{equation*}
\begin{split}
Y_{t}(\theta,\phi)&: \mathbb{R}\times\mathbb{R}\longrightarrow \mathbb{R}^{6}_{3} \\
 Y_{t}(\theta,\phi)&:=e_{1}+f_1.
 \end{split}
 \end{equation*}
 For simplicity, we omit the subscript $"_{t}"$ of $Y_{t}$. We have
that $y=[Y]:\mathbb{R}\times\mathbb{R}\rightarrow Q^{4}_{1}$ is a
timelike $(-)-$isothermic Willmore plane of $Q^{4}_{1}$ for any $0<t<1$.
Note that $y_t$ is also a conformally homogeneous surface in $Q^4_2$. We refer to \cite{Ma-W3} for details.
\end{example}

\section{Generalized c-polar transforms as Christoffel transforms}

\begin{definition}
Let $y:M\to Q^n_r$ be a timelike isothermic surface. Suppose that $\psi$ is a parallel
section of the normal bundle of $y$ with length $c$, i.e. $\psi\in\Gamma(V^{\bot})$, $D_u\psi=0$,
$D_v\psi=0$, and $\langle\psi,\psi\rangle=c$. Then we call $\psi:M\to N^{n+1}_r(\frac{1}{c})$
a c-polar transform of $y$ when $c\neq 0$, and $[\psi]:M\to Q^N_r$ a 0-polar transform of $y$
when $c=0$.
\end{definition}
Polar transforms can also be described as below:
\begin{proposition}
$\psi:M\to\R^{n+2}_{r+1}$ derives a c-polar transform of $y$ if it satisfies the following conditions :\\
(i). $\langle\psi,Y\rangle=\langle\psi,Y_u\rangle=\langle\psi,Y_v\rangle=0$;\\
(ii). $\psi_u\in\mathrm{Span}\{Y,Y_v\}$ and $\psi_v\in\mathrm{Span}\{Y,Y_u\}$.
\end{proposition}
\begin{proof}
From (ii) we see that $\langle\psi_u,Y_v\rangle=0$, together with (i),
yields $\langle\psi,Y_{uv}\rangle=0$. So $\psi$ is a parallel section
of the normal bundle by (ii). Also from (ii), $\langle\psi_u,\psi\rangle=\langle\psi_v,\psi\rangle=0$,
yielding $\langle\psi,\psi\rangle=c$ for some constant $c$.
\end{proof}

 Recall the definition of the classical Christoffel transforms. The Christoffel transform of a timelike $(\pm)-$isothermic surface $x:M\rightarrow \mathbb{R}^n_r$ is defined as a map $x^C:M\rightarrow \mathbb{R}^n_r$ such that $x^C_u\parallel x_{v}$ and $x^C_v\parallel x_{u}$ (see \cite{Bur}, \cite{Wangpeng}).  So $c-$polar transforms are generalizations of Christoffel transforms.

In \cite{Bur}, Burstall introduced the notion of   generalized $H$-surface to discuss properties of isothermic surfaces and to produce new examples. For timelike surfaces one can take verbatim from \cite{Bur}.
\begin{definition} Let $x:M\to N^n_r(c)$ be a  timelike surface with mean curvature vector $H$. $x$ is called a generalized $H$-surface if there exists a parallel isoperimetric section of $x$.
Here an isoperimetric section denotes a unit normal vector field $N$ of $x$ with $\langle N,H\rangle\equiv constant$. $N$ is called a minimal section when $\langle N,H\rangle\equiv 0$.
\end{definition}

\begin{theorem}
Let $y:M\to Q^n_r$ be a timelike $(\pm)-$isothermic surface, i.e. $\kappa_1=\pm \kappa_2$. We conclude that :\\
(i). Any c-polar transform $\psi:M_0\to N^{n+1}_r(\frac{1}{c})$
is also a timelike $(\pm)-$isothermic surface if $\langle\psi,\kappa_1\rangle\neq 0$ on an open dense subset $M_0$ of $M$.
And $\psi$ shares the same
adapted coordinate with $y$. Moreover, $\psi$ is a generalized $H$-surface admitting a parallel
minimal section. To be concrete, there exits a local lift $Y^{\psi}$ of $y$ (i.e., $[Y^{\psi}]=y$) with $Y_u^{\psi}\parallel\psi_v$,
$Y_v^{\psi}\parallel\psi_u$, that is , $Y^{\psi}$ is dual to $\psi$.\\
(ii). Let $\psi$ be a c-polar transform of $y$. The conformal invariant metric of $\psi$
is of the form
\begin{equation}
g^{\psi}=\left(\langle \kappa_1,\kappa_2\rangle+\left(\frac{\langle\psi,D_u \kappa_1}{\langle\psi,\kappa_1\rangle}\right)_v\right)dudv
\end{equation}
Furthermore, suppose that $M$ is a closed surface. If $\psi$ is globally immersed.
then $W(\psi)=W(y)$, i.e., in this case, the Willmore functional is polar transform invariant.
\end{theorem}
\begin{proof}
(i) Here we retain the notions in Section 2. Let $\psi\in\Gamma(V^{\perp})$ be a parallel section with $\langle\psi,\psi\rangle=c$. So we obtain that \begin{align*}
\psi_u=2\langle\psi,D_v \kappa_1\rangle Y-2\langle \psi,\kappa_1\rangle Y_v,\\
\psi_v=2\langle\psi,D_u \kappa_2\rangle Y-2\langle \psi,\kappa_2\rangle Y_u.
\end{align*}
Since $\kappa_2=\varepsilon \kappa_1$, $\varepsilon=\pm1$, we derive that
\begin{equation*}
\langle\psi_u,\psi_v\rangle=2\e\langle\psi,\kappa_1\rangle^2.
\end{equation*}
If $\langle\psi,\kappa_1\rangle\neq 0$ on an open dense subset $M_0$ of $M$, then
$\psi$ is an immersion. We calculate
\begin{align}
\psi_{uu}&=2\<\psi,D_u D_v \kappa_1\>Y-2\<\psi,D_u\kappa_1\>Y_v+2\<\psi,D_v\kappa_1\>Y_u-2\<\psi,\kappa_1\>Y_{vu}\nonumber\\
&=\frac{2\<D_u\kappa_1,\psi\>}{\<\psi,\kappa_1\>}\psi_u+\O_1^\psi
\end{align}
with
\begin{align}
\O_1^\psi&=-\<\psi,\kappa_1\>N+2\<\psi,D_v\kappa_1\>Y_u+2\<\psi,D_u\kappa_1\>Y_v\nonumber\\
&+2(\<\psi,D_uD_v\kappa_1\>+\<\psi,\kappa_1\>\<\kappa_1,\kappa_2\>-\frac{2\<\psi,D_u\kappa_1\>\<\psi,D_v\kappa_1\>}{\<\psi,\kappa_1\>})Y.
\end{align}
Similarly we get
\begin{align}
\psi_{vv}&=2\<\psi,D_v D_u \kappa_2\>Y-2\<\psi,D_v\kappa_2\>Y_u+2\<\psi,D_u\kappa_2\>Y_v-2\<\psi,\kappa_2\>Y_{uv}\nonumber\\
&=\frac{2\<D_v\kappa_2,\psi\>}{\<\psi,\kappa_2\>}\psi_v+\O_2^\psi
\end{align}
with
\begin{align}
\O_2^\psi&=-\<\psi,\kappa_2\>N+2\<\psi,D_u\kappa_2\>Y_v+2\<\psi,D_v\kappa_2\>Y_u\nonumber\\
&+2(\<\psi,D_vD_u\kappa_2\>+\<\psi,\kappa_2\>\<\kappa_1,\kappa_2\>-\frac{2\<\psi,D_u\kappa_2\>\<\psi,D_v\kappa_2\>}{\<\psi,\kappa_2\>})Y.
\end{align}
It follows from $\kappa_1=\e \kappa_2$ that
$$\O_1^\psi=\e\O_2^\psi,\hskip 5pt \e=\pm1.$$
Thus $\psi:M_0\to N^{n+1}_r(\frac{1}{c})$
is also a timelike $(\pm)-$ isothermic surface sharing the same
adapted coordinate with $y$.

For $\psi$ being a generalized $H$-surface, we first get
\begin{equation}
\psi_{uv}=-2\<\psi,\kappa_1\>\kappa_2+\<\psi,2D_vD_v\kappa_1+s_2\kappa_1\>Y.
\end{equation}
Let
\begin{equation*}
Y^\psi=\frac{1}{\<\psi,\kappa_1\>}Y.
\end{equation*}
We compute that
\begin{equation*}
Y^\psi_u=-\frac{1}{2\e\<\psi,\kappa_1\>^2}\psi_v,\hskip 5pt Y^\psi_v=-\frac{1}{2\<\psi,\kappa_1\>^2}\psi_u,
\end{equation*}
and
\begin{equation*}
\<Y^\psi,Y^\psi\>=\<Y^\psi,\psi\>=\<Y^\psi,\psi_u\>=\<Y^\psi,\psi_v\>=\<Y^\psi,\psi_{uv}\>=0,\Rightarrow
\<Y^\psi,H^\psi\>=0.
\end{equation*}
Here $H^\psi$ is the mean curvature vector of $\psi$. This finishes the proof of (i).

(ii). To the compute $g^\psi$, one derive directly by (15), (18) and (20) that
\begin{align*}
\<\kappa_1^\psi,\kappa_2^\psi\>&=\frac{1}{4\e\<\psi,\kappa_1\>^2}\<\O_1^\psi,\O_2^\psi\>\\
&=\<\kappa_1,\kappa_2\>+\frac{\<\psi,D_vD_u\kappa_1\>\<\psi,\kappa_1\>-\<\psi,D_u\kappa_1\>\<\psi,D_v\kappa_1\>}{\<\psi,\kappa_1\>^2}\\
&=\<\kappa_1,\kappa_2\>+\left(\frac{\<\psi,D_u\kappa_1\>}{\<\psi,\kappa_1\>}\right)_v
\end{align*}
For the Willmore functional, it is just a consequence of
\begin{equation*}
\left(\frac{\<\psi,D_u\kappa_1\>}{\<\psi,\kappa_1\>}\right)_vdu\wedge dv=-d\left(\frac{\<\psi,D_u\kappa_1\>}{\<\psi,\kappa_1\>}du\right)
\end{equation*}
by using the Stokes formula.
\end{proof}
\section{Timelike isothermic surfaces as an integrable system}

There are several equivalent methods to treat isothermic surfaces as integrable systems, see \cite{BDPT}, \cite{Bur}, \cite{BHPP}, \cite{CGS}. In each method, Darboux transform plays an essential role. So in this section, we first give a geometric definition of Darboux transforms for timelike isothermic surfaces. For a detailed discussion of Darboux transforms, we refer to \cite{Gu-H}. Then we introduce curved flat for timelike isothermic surfaces. In the end of this section we discuss briefly the relation between our treatment and the $O(n-r+1,r+1)/O(n-r,r)\times O(1,1)-$system II (in the sense of Terng, etc. \cite{BDPT}) methods which has been discussed in details by Dussan and Magid in \cite{Dussan-M2005}.

\subsection{Darboux transforms of timelike isothermic surfaces}

We define the Darboux transforms and give the basic properties of Darboux transforms as below (compare \cite{Dussan-M2005}, \cite{Ma-W2}, \cite{ZCC})
\begin{definition}
Let $y:M\to Q^n_r$ denote a timelike $(\pm)-$isothermic surface with canonical lift $Y$
with respect to the adapted coordinate $(u,v)$. A timelike immersion $\hat y:M\to Q^n_r$
is called a Darboux transform of $y$ if its local lift $\hat Y$ satisfies
\begin{equation}\label{eq-dar1}
\<Y,\hat Y\>\neq 0,\hskip 5pt \hat Y_u\in\mathrm{Span}\{\hat Y,Y,Y_v\}\hskip 3pt and \hskip 3pt\hat Y_v\in\mathrm{Span}\{\hat Y,Y,Y_u\}
\end{equation}
where $\hat Y$ is not necessarily the canonical lift,
and this definition is independent of the choice of local lift.
\end{definition}

\begin{remark}Geometrically, two isothermic surfaces $y$ and $\hat{y}$ in $Q^n_r$ form a
Darboux pair if they envelope the same sphere congruence at the
corresponding point and this transform preserves their
conformal curvature lines (see for example \cite{Gu-H}, \cite{BHPP}, \cite{BDPT}, \cite{Jer}, \cite{Ma-W2}, \cite{Dussan-M2005}, \cite{ZCC}, \cite{Wangpeng}).
\end{remark}
\begin{proposition}\label{Darboux}
Let $y:M\to Q^n_r$ be a timelike $(\pm)-$isothermic surface.
Assume that a timelike
immersion $\hat{y}:M\to Q^n_r$ is the Darboux transformation of $y$ with a lift $\hat{Y}$.
We have the following conclusions:
\begin{description}
    \item (i) $y$ and $\hat y$ share  the same adapted coordinate $(u,v)$ and they envelope the same 4-dimensional $(2,2)- $type subspace given by $\mathrm{Span}\{Y,\hat Y,dY\}=\mathrm{Span}\{Y,\hat
Y,d\hat Y\}$ (also called ``Lorentzian 2-sphere"
    congruence).
    \item (ii)   Set $\<Y,\hat Y>=-1$. We have
    \begin{equation}\label{eq-dar2}
\left\{\begin {array}{lllll}
\hat Y_{u}=b\hat Y+\theta_1 (Y_v+aY)\\
\hat Y_{v}=a\hat Y+\theta_2 (Y_u+bY)
\end {array}\right.
\end{equation}
     where $\theta_1=\pm\theta_2$ is a non-zero constant.
    \item (iii) $\hat Y$ is also a $(\pm)-$isothermic immersion sharing
    the same adapted coordinate $(u,v)$. As a consequence, the curvature lines of $y$ and $\hat y$ correspond to each other.
\end{description}
\end{proposition}
\begin{proof}
The conclusion (i) is obvious from the assumption \eqref{eq-dar1}. (Recall that a
round 2-sphere in $R^n_r$
is identified to a 4-dimensional Lorentzian subspace in $\R^{n+2}_{r+1}$.)
The normalization $\<Y,\hat Y\>=-1$ ensures $\<Y_u,\hat Y\>=-\<Y,\hat Y_u\>=b$ and
$\<Y_v,\hat Y\>=-\<Y,\hat Y_v\>=a$. Then
$\hat Y_u\in\mathrm{Span}\{\hat Y,Y,Y_v\}$ and $\hat Y_v\in\mathrm{Span}\{\hat Y,Y,Y_u\}$ is explicitly expressed by \eqref{eq-dar2}
Respectively differentiate these two formulas in \eqref{eq-dar2}, we obtain
\begin{equation}\label{eq-yuv}
\hat Y_{uv}=(b_v+ab)\hat Y+(b\t_2)Y_u+(\t_{1v}+a\t_1)Y_v+\t_1\kappa_2+(\cdot\cdot\cdot)Y
\end{equation}
and
\begin{equation}\label{eq-yvu}
\hat Y_{vu}=(a_u+ab)\hat Y+(a\t_1)Y_v+(\t_{2u}+b\t_2)Y_u+\t_2\kappa_1+(\cdot\cdot\cdot).
\end{equation}
Since $y$ is a timelike $(\pm)-$isothermic surface, $\kappa_1=\pm \kappa_2$. Comparing \eqref{eq-yuv} and
\eqref{eq-yvu}, we have the following
\begin{equation}\label{eq-darboux}
\left\{\begin {array}{lllll}
\t_1=\pm\t_2,\\
\t_{1v}=\t_{2u}=0,\\
a_u=b_v.
\end {array}\right.
\end{equation}
Hence $\t_1$ and $\t_2$ must be constant. They are both non-zero since $[\hat Y]$ is
an immersion. This verifies (ii).
\par
To show conclusion (iii) we need only to
show that $\hat Y^\bot_{uu}=\pm\hat Y^\bot_{vv}$. Differentiate \eqref{eq-dar2}, we obtain
\begin{align*}
\hat Y_{uu}&=\t_1(Y_{vu}+a_uY+aY_u)+b_u\hat Y+b\hat Y_u\\
&=\t_1Y_{vu}+(\t_1a_u)Y-(\t_1ab)Y\hskip 5pt \mod\ \{\hat Y,\hat Y_u,\hat Y_v\}
\end{align*}
and
\begin{align*}
\hat Y_{vv}&=\t_2(Y_{uv}+b_vY+bY_v)+a_v\hat Y+a\hat Y_v\\
&=\t_2Y_{uv}+(\t_2b_v)Y-(\t_2ab)Y\hskip 5pt \mod\ \{\hat Y,\hat Y_u,\hat Y_v\}
\end{align*}
It follows from (26) that $\hat Y^\bot_{uu}=\pm\hat Y^\bot_{vv}$. Thus
$\hat Y$ is a $(\pm)-$isothermic immersion sharing with the same adapted coordinate $(u,v)$,
and the u-curves and
v-curves are exactly the curvature lines on both of $y$ and $\hat y$. This completes the
proof.\\
\end{proof}

Following \cite{Bur}, we call $y$ and $\hat{y}=[\hat{y}]$ a Darboux pair of isothermic surfaces.

\subsection{Maurer-Cartan forms of a Darboux pair of timelike isothermic surfaces. }

 Let $y$ and $\hat{y}$ be a Darboux pair of isothermic surfaces. For the notion of curved flats, we need to deform the structure equations by using frames related with the Darboux pairs.

 To begin with, first we denote by $O(n-r+1,r+1)$ the group defined as below (compare \cite{BHPP})
\begin{equation}\left\{\begin {array}{lllll}
O(n-r+1,r+1):=\{A\in Mat(n+2,\mathbb{R})| A^t\hat{I}A=\hat I\},\\
\mathfrak{o}(n-r+1,r+1):=\{A\in  Mat(n+2,\mathbb{R}) | A^t\hat I+\hat IA=0\},
\\
\end {array}\right.\end{equation}
with
$$\hat I=\left(
           \begin{array}{ccc}
             J_2 &  &  \\
              & J_2 &  \\
              &  & I_{n-r-1,r-1} \\
           \end{array}
         \right).$$
         Here $$J_2=\left(
                                        \begin{array}{cc}
                                          0 & 1 \\
                                          1 & 0 \\
                                        \end{array}
                                      \right)\ \ \hbox{ and } \ \
I_{n-r-1,r-1}=diag\{1,\cdots,1,-1,\cdots,-1\}.$$

Now let us turn to the Darboux pairs. We can write out $\hat Y$ explicitly by using the frame of $y$:
\begin{equation}
\hat Y=N+2aY_u+2bY_v+(2ab+\frac{1}{2}\<\xi,\xi\>)Y+\xi.
\end{equation}
Set
\begin{equation}
P_1=Y_v+aY \ \hbox{ and }\  P_2=Y_u+bY.
\end{equation}
The structure equations \eqref{eq-moving1} of Y can be rewritten with respect to the frame
$\{Y,\hat Y,P_1,P_2,\psi\}$ as below
\begin{equation}\label{eq-moving2}
\left\{\begin {array}{lllll}
\hat Y_u=b\hat Y+\t_1P_1,\\
\hat Y_v=a\hat Y+\t_2P_2, \\
P_{1u}=-bP_1+\frac{1}{2}\hat Y-\frac{1}{2}(\xi+\<\xi,\xi\>Y),\\
P_{1v}=aP_1+\frac{\t_2}{2}Y+(\kappa_2+\<\kappa_2,\xi\>Y),\\
P_{2u}=bP_2+\frac{\t_1}{2}Y+(\kappa_1+\<\kappa_1,\xi\>Y),\\
P_{2v}=-aP_2+\frac{1}{2}\hat Y-\frac{1}{2}\<\xi,\xi\>Y-\frac{1}{2}\xi,\\
\psi_u=-2\<\psi,\kappa_1\>P_1-\<\psi,D_u\xi-b\xi\>Y,\\
\psi_v=-2\<\psi,\kappa_2\>P_2-\<\psi,D_v\xi-a\xi\>Y,\\
\end {array}\right.
\end{equation}
Let $\tilde{\psi}_j=\psi_j+\langle\psi_j,\xi\rangle Y$, we have
$$\left\{\begin {array}{lllll}
\tilde{\psi}_{ju}=-2\<\tilde\psi_{j},\kappa_1\>P_1+\<\tilde\psi_{j},\xi\>P_2,\\
\tilde{\psi}_{jv}=-2\<\tilde\psi_{j},\kappa_2\>P_2+\<\tilde\psi_{j},\xi\>P_1.\\
\end {array}\right.$$
Set $F=\left(Y,-\hat{Y},P_1,2P_2,\tilde\psi_3,\cdots,\tilde{\psi}_{n}\right)$, and assume that
$$k_{1j}=\<\kappa_1,\psi_j\>,\ k_{2j}=\<\kappa_2,\psi_j\>,\ \epsilon_j=\<\psi_j,\psi_j\>, \ c_{ j}=\<\xi,\psi_j\>,\  j=3,\cdots,n. $$
Note that $\epsilon_j=1,\ 3\leq j\leq n-r+1,\ \epsilon_j=-1,\ n-r+2\leq j\leq n. $ Then we obtain that
\begin{equation}
\alpha=F^{-1}F_u=\left(
                           \begin{array}{cccccccc}
                            -b & 0   & 0  &    \theta_1&  0 &   \cdots & 0 \\
                            0  &b & -\frac{1}{2} &   0  &  0 &   \cdots & 0  \\
                            0 & -\theta_1  & -b  &   0 &  - k_{13}\epsilon_3 &   \cdots & - k_{1n}\epsilon_n  \\
                            \frac{1}{2}  &  0  & 0  &   b  &  \frac{1}{2}c_3\epsilon_3 &   \cdots & \frac{1}{2}c_n\epsilon_n\\
                            0  &0  & -\frac{1}{2}c_3 &  2 k_{13}  &  0 &   \cdots & 0 \\
                            \cdots &  \cdots   &  \cdots    &  \cdots  &  \cdots   &  \cdots  &  \cdots    \\
                            0  &0 & -\frac{1}{2}c_n  &   2k_{1n}  &  0 &   \cdots & 0  \\
                             \end{array}
                         \right),
\end{equation}
and
\begin{equation}
\beta=F^{-1}F_v=\left(
                           \begin{array}{cccccccc}
                            -a  & 0  & \frac{1}{2}\theta_2 &   0 &   0 &   \cdots & 0 \\
                            0 &   a &   0 & -1   &  0 &   \cdots & 0   \\
                            1  &  0  & a &  0 &  c_3\epsilon_3 &   \cdots & c_n\epsilon_n  \\
                            0 &-\frac{1}{2} \theta_2  & 0&   -a  &    - k_{23}\epsilon_3 &   \cdots & - k_{2n}\epsilon_n \\
                            0   &0 &   k_{23} & - c_3  &  0 &   \cdots & 0  \\
                            \cdots &  \cdots    &  \cdots    &  \cdots  &  \cdots   &  \cdots  &  \cdots   \\
                            0   &0&  k_{2n}  &  - c_n &  0 &   \cdots & 0  \\
                             \end{array}
                         \right).
\end{equation}
Note that now $F$ provides a frame in $O(n-r+1,r+1)$.

Assume moreover that
\begin{equation}F^{-1}dF= \alpha du+ \beta dv=(\alpha_0+\alpha_1)du+(\beta_0+\beta_1)dv.
\end{equation}
with
 \begin{equation} \label{eq-A1} \alpha_0=\left(
                           \begin{array}{cccccccc}
                            -b & 0   & 0  &  0&  0 &   \cdots & 0 \\
                            0  &b & 0 &   0  &  0 &   \cdots & 0  \\
                            0 &0  & -b  &   0 &  - k_{13}\epsilon_3 &   \cdots & - k_{1n}\epsilon_n  \\
                            0  &  0  & 0  &   b  &  \frac{1}{2}c_3\epsilon_3 &   \cdots & \frac{1}{2}c_n\epsilon_n\\
                            0  &0  & -\frac{1}{2}c_3 &  2 k_{13}  &  0 &   \cdots & 0 \\
                            \cdots &  \cdots   &  \cdots    &  \cdots  &  \cdots   &  \cdots  &  \cdots    \\
                            0  &0 & -\frac{1}{2}c_n  &   2k_{1n}  &  0 &   \cdots & 0  \\
                             \end{array}
                         \right),
\end{equation}
and \begin{equation} \label{eq-A2}  \alpha_1=\left(
                           \begin{array}{cccccccc}
                            0& 0   & 0  &    \theta_1&  0 &   \cdots & 0 \\
                            0  &0 & -\frac{1}{2} &   0  &  0 &   \cdots & 0  \\
                            0 & -\theta_1  & 0&   0 & 0 &   \cdots & 0 \\
                           \frac{ 1}{2}  &  0  & 0  & 0  & 0 &   \cdots &0\\
                            0  &0  & 0&  0 &  0 &   \cdots & 0 \\
                            \cdots &  \cdots   &  \cdots    &  \cdots  &  \cdots   &  \cdots  &  \cdots    \\
                            0  &0 &  0  &  0 &  0 &   \cdots & 0  \\
                             \end{array}
                         \right),
\end{equation}
and
\begin{equation}\label{eq-B1} \beta_0=\left(
                           \begin{array}{cccccccc}
                            -a  & 0  &0 &   0 &   0 &   \cdots & 0 \\
                            0 &   a &   0 & 0  &  0 &   \cdots & 0   \\
                            0  &  0  & a &  0 &  c_3\epsilon_3 &   \cdots & c_n\epsilon_n  \\
                            0 & 0 & 0&   -a  &    - k_{23}\epsilon_3 &   \cdots & - k_{2n}\epsilon_n \\
                            0   &0 &   k_{23} & - c_3  &  0 &   \cdots & 0  \\
                            \cdots &  \cdots    &  \cdots    &  \cdots  &  \cdots   &  \cdots  &  \cdots   \\
                            0   &0&  k_{2n}  &  - c_n &  0 &   \cdots & 0  \\
                             \end{array}
                         \right)
\end{equation}
and
\begin{equation}\label{eq-B2}  \beta_1=\left(
                           \begin{array}{cccccccc}
                            0 & 0  & \frac{1}{2}\theta_2 &   0 &   0 &   \cdots & 0 \\
                            0 &  0 &   0 &- 1   &  0 &   \cdots & 0   \\
                            1  &  0  &0&  0 & 0 &   \cdots & 0 \\
                            0 & -\frac{1}{2}\theta_2  & 0& 0  &    0 &   \cdots & 0\\
                            0   &0 &  0& 0  &  0 &   \cdots & 0  \\
                            \cdots &  \cdots    &  \cdots    &  \cdots  &  \cdots   &  \cdots  &  \cdots   \\
                            0   &0 &  0& 0  &  0 &   \cdots & 0  \\
                             \end{array}
                         \right).
\end{equation}
It is straightforward to verify that $$\alpha_1\beta  _1-\beta  _1\alpha_1=0.$$
On the other hand, the integrability of $F$ yields
$$\alpha _{v}-\beta  _{u}-\alpha \beta  +\beta  \alpha =0.$$
Altogether we derive that
\begin{equation}\label{eq-darboux2}\left\{\begin{split}
&\alpha _{0v}-\beta  _{0u}-\alpha _0\beta  _0+\beta  _0\alpha _0=0,\\
&\alpha _{1v}-\beta  _{1u}-\alpha _1\beta  _0+\beta  _0\alpha _1-\alpha _0\beta  _1+\beta  _1\alpha _0=0,\\
&\alpha _1\beta  _1-\beta  _1\alpha _1=0.\\
\end{split}\right.
\end{equation}

\subsection{Timelike isothermic surfaces as Curved flats}

One of the basic relations of isothermic surfaces with integrable system involves the so-called curved flats discovered
by Ferus-Pedit \cite{FP}. In \cite{BHPP}, they showed that an isothermic surface in $\mathbb{R}^3$ together with its Darboux transform defines a curved flat in $SO(1,4)/SO(3)\times SO(1,1)$. For the higher co-dimensional case, see Burstall's summary paper on isothermic surfaces \cite{Bur}.  Here we will show that timelike isothermic surface together with its Darboux transform can also be related with a curved flat.

First we recall the definition of curved flats. Let $G/K$ be a symmetric space with an involution $\sigma:G\rightarrow G$ of (semi-simple) Lie group $G$ and $K$ as the fixed subgroup of $\sigma$. Then we have the decomposition
$$\mathfrak{g}=\mathfrak{k}\oplus\mathfrak{p}$$
with $\mathfrak{k}$ the $+1$-eigenspace of $\sigma$, and $\mathfrak{p}$ the $-1$-eigenspace of $\sigma$. Note the famous conditions
$$[\mathfrak{k},\mathfrak{k}]\subset\mathfrak{k},\ [\mathfrak{p},\mathfrak{p}]\subset\mathfrak{k},\ [\mathfrak{k},\mathfrak{p}]\subset\mathfrak{p}.$$
Then we have
\begin{definition} \cite{BHPP}, \cite{FP} Let $\phi: M\rightarrow G/K$ be an immersion with a frame $F:M \rightarrow G$. Suppose the Maurer-Cartan form of $F$ is decomposed according $\sigma$ as
$$\Phi=F^{-1}dF=\alpha_{\mathfrak{k}}+\alpha_{\mathfrak{p}}.$$
Then $\phi$ is called a curved flat, if each $\alpha_{\mathfrak{p}}$ is contained in an Abelian subalgebra of $\mathfrak{g}$.
\end{definition}

 It is well-known that one can introduce a loop of Maurer-Cartan forms as follows (Lemma 3 in \cite{FP}, Proposition 3.1 \cite{Bur}):
\begin{proposition} Let $F : M\rightarrow G$ with $F^{-1}dF = \alpha_{\mathfrak{k}}+\alpha_{\mathfrak{p}}$. Then $F$ frames a
curved flat if and only if

\begin{equation}
\alpha_{\lambda}=\alpha_{\mathfrak{k}}+\lambda\alpha_{\mathfrak{p}}
\end{equation} satisfies
\begin{equation}
d\alpha_{\lambda}+\frac{1}{2}[\alpha_{\lambda} \wedge \alpha_{\lambda}]=0,\ \forall \lambda\in \mathbb{R}.
\end{equation}
\end{proposition}
For the existence of $F_{\lambda}$ framing a curved flat $f$, we refer to Theorem 3.2 \cite{Bur}.

For a Darboux pair of timelike isothermic surfaces, it is obvious that the related $G/K$ is described as
 $$G=O(n-r+1,r+1), \  $$
and \begin{equation}K=O(n-r,r)\times O(1,1)=\left\{A\in G\left|\right. \left(
                                                          \begin{array}{ccc}
                                                            -I_2 & 0  \\
                                                            0   & I_n \\
                                                          \end{array}
                                                        \right)\cdot A\cdot\left(
                                                          \begin{array}{ccc}
                                                            -I_2 & 0  \\
                                                            0   & I_n \\
                                                          \end{array}
                                                        \right)=A
  \right\}.\end{equation}
  So
 \begin{equation}\mathfrak{k}=Lie K=\left\{
                                     \left(
                                                          \begin{array}{ccc}
                                                             A_1 &0 \\
                                                            0 & A_2\\
                                                          \end{array}
                                                        \right)\left|\right.\ A_1^tJ_2+J_2A_1=0,\ A_2^t\tilde{I}+\tilde{I}A_2=0\right\} \subset\mathfrak{o}(n-r+1,r+1),\end{equation}
 and \begin{equation}\mathfrak{p}=\left\{
                                     \left(
                                                          \begin{array}{ccc}
                                                            0 & -b_1^t\tilde{I}\\
                                                            b_1 & 0  \\
                                                          \end{array}
                                                        \right) \right\} \subset\mathfrak{o}(n-r+1,r+1).
  \end{equation} Here $$\tilde{I}= \left(
                                     \begin{array}{cc}
                                       J_2 & 0 \\
                                     0& I_{n-r-1,r-1} \\
                                     \end{array}
                                   \right).
  $$
 Looking $O(n-r+1,r+1)/O(n-r,r)\times O(1,1)$ as a Grassmannian of two-dimensional Lorentzian plane in $\mathbb{R}^{n+2}_{r+1}$, the map $f=y\wedge\hat{y}: M\rightarrow O(n-r+1,r+1)/O(n-r,r)\times O(1,1)$ defines a pair of surfaces $y,\hat{y}:M \rightarrow Q^n_r$. $f$ is called {\em Lorentzian non-degenerate} if both $y$ and $\hat{y}$ are Lorentzian immersions.
 We have a timelike version of the main theorem in \cite{BHPP}, and Theorem 3.3 in \cite{Bur}:

\begin{theorem}\label{th-curveflat} The Lorentzian non-degenerate map $f=y\wedge\hat{y}: M\rightarrow O(n-r+1,r+1)/O(n-r,r)\times O(1,1)$ is a
curved flat if and only if $y$ and $\hat{y}$ is a Darboux pair of timelike $(\pm)-$isothermic surfaces.
\end{theorem}
\begin{proof}For a Darboux pair $(y,\hat{y})$ as described in above subsection, we have a frame $F$ such that
$$F^{-1}dF=\alpha du+\beta  dv=(\alpha _0+\alpha _1)du+(\beta  _0+\beta  _1)dv,$$
with $\alpha _0$, $\alpha _1$, $\beta  _0$ and $\beta  _1$ of the form in \eqref{eq-A1}, \eqref{eq-A2}, \eqref{eq-B1} and \eqref{eq-B2}. Set
$$\alpha _\lambda=\alpha _0+\lambda\alpha _1\ \hbox{ and } \ \beta  _\lambda=\beta  _0+\lambda\beta  _1$$
with $\lambda\in \mathbb{R}$.
One verifies directly that
$$\alpha _{\lambda v}-\beta  _{\lambda u}-\alpha _{\lambda}\beta  _{\lambda}+\beta  _{\lambda}\alpha _{\lambda}=0$$
is equivalent to the equations in \eqref{eq-darboux2}. As a consequence, $F$ induces a curved flat in $G/K$.

Conversely, let $f=y\wedge\hat{y}$ be a Lorentzian non-degenerate curved flat. Since $y$ inherits a Lorentzian metric, there exists some local  asymptotic coordinate $(u,v)$ and some lift $Y$ of $y$ such that $\<Y_u,Y_u\>=\<Y_v,Y_v\>=0,\ \<Y_u,Y_v\>=\frac{1}{2}$. Let $\hat{Y}$ be a lift of $\hat{y}$ such that $\langle\hat{Y},Y \rangle=-1$. Set \begin{equation*}
P_1=Y_v+aY \ \hbox{ and }\  P_2=Y_u+bY
\end{equation*}
as above such that $P_1\perp \hat{Y}$ and  $P_2\perp \hat{Y}$.
Choose $\tilde{\psi}_j$ such that
$$F=(Y,-\hat{Y},P_1,2P_2,\tilde{\psi}_3,\cdots,\tilde{\psi}_n)$$
represents a map into $O(n-r+1,r+1)$ (Hence $\epsilon_j=\<\psi_j,\psi_j\>$ with $\ \epsilon_j=1,\ 3\leq j\leq n-r+1,\ \epsilon_j=-1,\ n-r+2\leq j\leq n$). Then we may assume that
$$\hat{Y}_u=a\hat{Y}+p_0P_1+\tilde{p}_0P_2+\sum_{j=3}^n p_j\tilde{\psi}_j,\ \ \hat{Y}_v=b\hat{Y}+q_0P_1+\tilde{q}_0P_2+\sum_{j=3}^n q_j\tilde{\psi}_j.$$
Now set
$$F^{-1}dF= \alpha du+ \beta dv=(\alpha_0+\alpha_1)du+(\beta_0+\beta_1)dv$$
with $\alpha_0,\beta_0\in\mathfrak{k}$ and  $\alpha_1,\beta_1\in\mathfrak{p}$. Then we obtain that
\begin{equation*} \alpha_1=\left(
                           \begin{array}{cccccccc}
                            0& 0   & \frac{1}{2}\tilde{p}_0   &    p_0 &  p_3\epsilon_3 &   \cdots & p_n\epsilon_n \\
                            0  &0 & -\frac{1}{2} &   0  &  0 &   \cdots & 0  \\
                            0 & -p_0  & 0&   0 & 0 &   \cdots & 0 \\
                           \frac{ 1}{2}  &  -\frac{1}{2}\tilde{p}_0  & 0  & 0  & 0 &   \cdots &0\\
                            0  & -p_3 & 0&  0 &  0 &   \cdots & 0 \\
                            \cdots &  \cdots   &  \cdots    &  \cdots  &  \cdots   &  \cdots  &  \cdots    \\
                            0  & -p_n &  0  &  0 &  0 &   \cdots & 0  \\
                             \end{array}
                         \right)
\end{equation*}
and
\begin{equation*}  \beta_1=\left(
                           \begin{array}{cccccccc}
                            0 & 0   & \frac{1}{2}\tilde{q}_0 & q_0 &   q_3\epsilon_3 &   \cdots & q_n\epsilon_n \\
                            0 &  0 &    0 & - 1   &  0 &   \cdots & 0   \\
                            1  &  -q_0  &0&  0 & 0 &   \cdots & 0 \\
                            0 & -\frac{1}{2}\tilde{q}_0  & 0& 0  &    0 &   \cdots & 0\\
                            0   & -q_3 &  0& 0  &  0 &   \cdots & 0  \\
                            \cdots &  \cdots    &  \cdots    &  \cdots  &  \cdots   &  \cdots  &  \cdots   \\
                            0   &-q_n &  0& 0  &  0 &   \cdots & 0  \\
                             \end{array}
                         \right).
\end{equation*}
The conditions of curved flats are equivalent to \eqref{eq-darboux2}. Now the last matrix equation of \eqref{eq-darboux2} yields
$$p_3=\cdots=p_n=q_3=\cdots=q_n, \ \tilde{p}_0+q_0=0, \ \tilde{p}_0-q_0=0.$$
Therefore we obtain $$\hat{Y}_u\in span\{Y,\hat{Y},Y_v\},\ \hat{Y}_v\in span\{Y,\hat{Y},Y_u\}\ \hbox{ and } \ \<\hat{Y}_u,\hat{Y}_u\>=\<\hat{Y}_v,\hat{Y}_v\>=0.$$
Since $\hat{Y}$ is non-degenerate, $ \<\hat{Y}_u,\hat{Y}_v\>\neq 0$, by Theorem A of \cite{Wangpeng} (compare Proposition \ref{Darboux}), $[Y]$ and $[\hat{Y}]$ are
 a Darboux pair of timelike $(\pm)-$isothermic surfaces.
\end{proof}

\begin{remark}
 Note that when $\<\hat{Y}_u,\hat{Y}_v\> >0$, one obtains a pair of timelike $(+)-$isothermic surfaces, when $\<\hat{Y}_u,\hat{Y}_v\> <0$, one obtains a pair of timelike $(-)-$isothermic surfaces.
 \end{remark}

\begin{remark} It was explained in \cite{BHPP} that
 Christoffel transforms can be obtained as a limit of Darboux transforms for isothermic surfaces in $\mathbb{R}^3$.
 While for $c-$polar transforms, usually it can not be derived as a limit of Darboux transforms. Therefore the classical Christoffel transforms can be looked as a very special kind of transforms of  $c-$polar transforms which have closed relation with Darboux transforms.
 \end{remark}

\begin{remark}{\em $O(n-j+1,j+1/O(n-j,j))\times O(1,1)-$system and Timelike isothermic surfaces}:

Along an independent line, almost in the same time, Terng introduced the notion of $U/K-$system \cite{Terng}, which also has its roots on the study of transforms of submanifolds(\cite{TT}). In \cite{Terng}, Terng showed that the $U/K-$system, as an integrable system, inherits Lax pair and one also can introduce loop parameter $\lambda$, etc. Moreover, such systems are naturally related with several kinds of submanifolds together with their transforms. Such relationships are discussed in details in the later note \cite{BDPT}. As a special case, in \cite{BDPT}, they found that $G^1_{m,1}-$system ($O(m+1,1)/O(m)\times O(1,1)$) is exactly related with a pair of dual isothermic surfaces in $\mathbb{R}^m$. See Section 8 \cite{BDPT} for details.

In the recent work by Dussan and Magid \cite{Dussan-M2005}, they generalized the methods in \cite{BDPT} to derive a detailed description on timelike isothermic surfaces in $\mathbb{R}^n_j$. Comparing with the formulas in Theorem 3.1 and Theorem 3.2 in \cite{Dussan-M2005}, one can see that these are exactly the same as the integrability conditions in \eqref{eq-inte1}, \eqref{eq-inte2}, \eqref{eq-inte3} for the case $(+)-$isothermic surfaces. And the formulas in Theorem 3.1 and Theorem 3.2 in \cite{Dussan-M2005}, one can see that these are exactly the same as the integrability conditions in \eqref{eq-inte1}, \eqref{eq-inte2}, \eqref{eq-inte3} for the case $(-)-$isothermic surfaces. To get Darboux transforms, they need the methods of dressing actions, see \cite{BDPT}, \cite{Dussan-M2006}.

 \end{remark}

\section{Two permutability theorems}
\subsection{Spectral transforms of timelike isothermic surfaces}
Let $y:M\to Q^n_r$ be an immersed timelike $(+)-$isothermic surface with adapt coordinate $(u,v)$ and invariants in Section 2. If one only change the Schwarzians from $s_i$ to $s^{\tilde c}_i=s_i+\tilde c$, with all the other coefficients invariant, then  the conformal Gauss, Codazzi, and
Ricci equations are remain satisfied, where $\tilde c\in \R$ is a parameter.

To be concrete, if we consider a new data as below
\begin{align*}
&s^{\tilde c}_i=s_i+\tilde c,\hskip 5pt \<k^{\tilde c}_i,\psi^{\tilde c}\>=\<k_i,\psi\>,\hskip 5pt
\<\psi^{\tilde c},\psi^{\tilde c}\>=\<\psi,\psi\>,\hskip 5pt D^{\tilde c}_z=D_z.
\end{align*}
where $\psi$ is arbitrary section of normal bundle with deforming
normal section $\psi^{\tilde c}$. The integrable equations are satisfied and then by Theorem \ref{Bonnet}, there will be an associate family
of non-congruent timelike  $(+)-$isothermic surfaces $[Y^{\tilde c}]$ with corresponding invariants.

Note that for a timelike $(-)-$isothermic surface, the deformation should be  $s_1$ to $s^{\tilde c}_1=s_1+\tilde c$ and  $s_2$ to $s^{\tilde c}_2=s_2-\tilde c$.

\begin{definition}  The timelike isothermic surface  $[Y^{\tilde c}]$ is called a ($\tilde{c}-$parameter) spectral
transform of the  timelike isothermic surface $y:M\to Q^n_r$.
\end{definition}

\begin{theorem}
Let $y^{\tilde c}$ be a $\tilde{c}-$parameter spectral transform of a timelike $(\pm)-$ isothermic surface $y:M\to Q^n_r$.
Denote their canonical lift as $Y$, $Y^{\tilde c}$
for the same adapted coordinate $(u,v)$. Let $\psi$ be a non-degenerate $c-$polar surface of $y$. Then there exists
a $c-$polar surface $\psi^{\tilde c}$ of $y^{\tilde c}$ satisfying that $\psi^{\tilde c}$ is a $\tilde{c}-$parameter spectral
transform of $\psi$.  In other words, one has the commuting diagram:
\[
\begin{array}{cccc}
[Y] & \longrightarrow & [Y^{\tilde c}] \\
\downarrow  &\ & \downarrow\\
\psi  & \longrightarrow & \psi^{\tilde c}
\end{array}
\]
\end{theorem}
\begin{proof}
 Note that the spectral transforms preserve normal connection and the corresponding $k_i$. Let $\psi^{\tilde c}$ denote
the corresponding parallel normal section of $\psi$, by Theorem 3.4, $\psi^{\tilde c}$ is also a $c-$polar
 transform of $y^{\tilde c}$. By using (15), (17)-(21), all the other conditions except Schwarzian
 are satisfied.  For $s_i$, we have that
\begin{equation*}
 \left\{\begin {array}{lllll}
s_1^{\psi}=2\o^\psi_{uu}-2(\o^\psi_u)^2+\frac{\<\psi,D_u D_u \kappa_2\>}{\<\psi,\kappa_2\>}+s_1,\\
s_2^{\psi}=2\o^\psi_{vv}-2(\o^\psi_v)^2+\frac{\<\psi,D_v D_v \kappa_1\>}{\<\psi,\kappa_1\>}+s_2.
\end {array}\right.
\end{equation*}
So $s^{\tilde c}_i=s_i+\tilde c$ when $y$ is $(+)-$isothermic and $\psi^{\tilde c}$ is
a $\tilde{c}-$parameter spectral transform of $\psi$. When $y$ is $(-)-$isothermic, we have $s^{\tilde c}_1=s_1+\tilde c$ , $s^{\tilde c}_2=s_2-\tilde c$ and  $\psi^{\tilde c}$ is also
a $\tilde{c}-$parameter spectral transform of $\psi$.
\end{proof}

\subsection{Permutability with Darboux transforms}

\begin{theorem}
Let $y:M\to Q^n_r$ be a timelike $(\pm)-$isothermic surface and $[\hat Y]$ be a
Darboux transform of $y$. If $\psi$ is a non-degenerate c-polar surface of $y$, then
their exits a c-polar surface $\hat\psi$ of $\hat y$ such that $\hat\psi$ is also a Darboux
transform of $\psi$.
\end{theorem}

\begin{proof}

Write out $\hat Y$ explicitly:
\begin{equation}
\hat Y=N+2aY_u+2bY_v+(2ab+\frac{1}{2}\<\xi,\xi\>)Y+\xi.
\end{equation}
Set
\begin{equation}
P_1=Y_v+aY,\hskip 5pt P_2=Y_u+bY
\end{equation}
The structure equations (14) of Y can be rewritten with respect to the frame
$\{Y,\hat Y,P_1,P_2,\psi\}$ as below
\begin{equation}\label{eq-moving6}
\left\{\begin {array}{lllll}
\hat Y_u=b\hat Y+\t_1P_1,\\
\hat Y_v=a\hat Y+\t_2P_2, \\
P_{1u}=-bP_1+\frac{1}{2}\hat Y-\frac{1}{2}\<\xi,\xi\>Y-\frac{1}{2}\xi,\\
P_{1v}=aP_1+(\frac{\t_2}{2}+\<\kappa_2,\xi\>)Y+\kappa_2,\\
P_{2u}=bP_2+(\frac{\t_1}{2}+\<\kappa_1,\xi\>)Y+\kappa_1,\\
P_{2v}=-aP_2+\frac{1}{2}\hat Y-\frac{1}{2}\<\xi,\xi\>Y-\frac{1}{2}\xi,\\
\psi_u=-2\<\psi,\kappa_1\>P_1-\<\psi,D_u\xi-b\xi\>Y,\\
\psi_v=-2\<\psi,\kappa_2\>P_2-\<\psi,D_v\xi-a\xi\>Y,\\
\end {array}\right.
\end{equation}
Now let us find out the $c-$polar transform of $\hat Y$ corresponding to $\psi$. From (29)
we have
\begin{equation*}
\hat Y_{uv}=-\frac{1}{4}\<\xi,\xi\>\hat Y +\frac{\t_1\t_2}{2}\hat N
\end{equation*}
where
\begin{equation*}
\hat N=(1+\frac{2\<\kappa_2,\xi\>}{\t_2})Y+\frac{2a}{\t_2}P_1+\frac{2b}{\t_1}P_2+
2(\frac{ab}{\t_1\t_2}+\frac{2\<\kappa_1,\kappa_2\>}{\t_1\t_2})\hat Y+\frac{2}{\t_2}\kappa_2.
\end{equation*}
Denote
\begin{equation}
\hat\psi=\psi+\<\psi,\xi\>Y+\frac{2\<\psi,\kappa_2\>}{\t_2}\hat Y.
\end{equation}
It is straightforward to verify that
\begin{align*}
\hat\psi\perp\{\hat Y,\hat Y_u,\hat Y_v,\hat Y_{uv}\},
\end{align*}
and $\hat\psi$ is a parallel section of $\hat Y$ with length $c$.
Computing shows
\begin{equation}
\hat\psi_u=-\frac{\<\psi,\xi\>}{2\<\psi,\kappa_2\>}\psi_v+\left(\frac{\<\psi,D_u\kappa_2\>}{\<\psi,\kappa_2\>}+b\right)(\hat\psi-\psi),
\end{equation}
and
\begin{equation}
\hat\psi_v=-\frac{\<\psi,\xi\>}{2\<\psi,\kappa_1\>}\psi_u+\left(\frac{\<\psi,D_v\kappa_2\>}{\<\psi,\kappa_2\>}+a\right)(\hat\psi-\psi).
\end{equation}
This shows that $\hat\psi$ is a Darboux transform of $\psi$ by Definition 4.1.
\end{proof}



\begin{flushleft}
Yuping Song\\
LMAM\\School of Mathematical Sciences\\Xiamen University
\\Xiamen
361005,
\\ P. R. China\\
\texttt{ypsong@xmu.edu.cn}
\end{flushleft}
\begin{flushleft}
Peng Wang\\
Department of Mathematics,
\\Tongji University
\\Shanghai, 200092,\\ P. R. China\\
\texttt{netwangpeng@tongji.edu.cn}
\end{flushleft}
\end{document}